\documentclass[12pt]{amsart}
\usepackage{amsmath,amsthm,amssymb,amsfonts}
\usepackage{latexsym}
\usepackage[colorlinks=true,linkcolor=blue,citecolor=blue]{hyperref}

\DeclareMathOperator{\RE}{Re} 		\DeclareMathOperator{\sign}{sign}
\DeclareMathOperator{\IM}{Im} 		
\DeclareMathOperator{\vspan}{span} 	
\DeclareMathOperator{\ext}{ext}

\DeclareMathOperator{\co}{co}

\theoremstyle{definition}
\newtheorem{defn}{Definition}[section]
\theoremstyle{plain}
\theoremstyle{definition}
\newtheorem{rem}[defn]{Remark}

\theoremstyle{plain}
\newtheorem{lem}[defn]{Lemma}
\newtheorem{thm}[defn]{Theorem}

\theoremstyle{plain}
\newtheorem{prop}[defn]{Proposition}
\theoremstyle{plain}
\newtheorem{cor}[defn]{Corollary}

\newcommand{\R}{\mathbb{R}}
\newcommand{\C}{\mathbb{C}}
\newcommand{\eps}{\varepsilon}

\renewcommand{\leq}{\leqslant}
\renewcommand{\geq}{\geqslant}

\numberwithin{equation}{section}

\begin{document}
\title[Strong peak points and Strongly norm attaining points]{Strong peak points and strongly norm attaining points with applications to denseness and polynomial numerical indices}
\date{}

\author[J.~Kim]{Jaegil Kim}
\author[H.~J.~Lee]{Han Ju Lee}


\baselineskip=.6cm

\begin{abstract}
Using the variational method, it is shown that the set of all strong peak functions in a closed algebra $A$ of $C_b(K)$ is dense if and only if the set of all strong peak points is a norming subset of $A$. As a corollary we can induce the denseness of strong peak functions on other certain spaces. In case that a set of uniformly strongly exposed points of a Banach space $X$ is a norming subset of $\mathcal{P}({}^n X)$, then the set of all strongly norm attaining elements in $\mathcal{P}({}^n X)$ is dense. In particular, the set of all points at which the norm of $\mathcal{P}({}^n X)$ is Fr\'echet differentiable is a dense $G_\delta$ subset.

In the last part, using Reisner's graph theoretic-approach, we construct some strongly norm attaining polynomials on a CL-space with an absolute norm.
Then we show that for a finite dimensional complex Banach space $X$ with an absolute norm, its polynomial numerical indices are one if and only if $X$ is isometric to $\ell_\infty^n$. Moreover, we give a characterization of the set of all complex extreme points of the unit ball of a CL-space with an absolute norm.
\end{abstract}

\maketitle

\section{Introduction and Preliminaries}

Let $K$ be a complete metric space and $X$ a (real or complex) Banach space. We denote by $C_b(K:X)$ the Banach space of all bounded continuous functions from $K$ to $X$ with the supremum norm. A nonzero function $f\in C_b(K:X)$ is said to be a {\it strong peak function} at $t\in K$ if every sequence $\{t_n\}$ in $K$ with $\lim_n\|f(t_n)\|=\|f\|$ converges to $t$. Given a subspace $A$ of $C_b(K:X)$, a point $t\in K$ called a {\it strong peak point} for $A$ if there is a strong peak function $f$ in $A$ with $\|f\|= \|f(t)\|$. We denote by $\rho A$ the set of all strong peak points for $A$.

Let $B_X$ (resp. $S_X$) be the unit ball (resp. sphere) of a Banach space $X$. A nonzero function $f \in C_b(B_X:Y)$ is said to \emph{strongly attain its norm} at $x$ if for every sequence $\{x_n\}$ in $B_X$ with $\lim_n\|f(x_n)\|=\|f\|$, there exist a scalar $\lambda$ with $|\lambda|=1$ and a subsequence of $\{x_n\}$ which converges to $\lambda x$. Given a subspace $A$ of $C_b(B_X:Y)$, $x \in B_X$ is called a \emph{strongly norm-attaining point} of $A$ if there exists a nonzero function $f$ in $A$ which strongly attain its norm at $x$. Denote by $\tilde{\rho}A$ the set of all strongly norm-attaining points of $A$.

For complex Banach spaces $X$ and $Y$, we may use the following two subspaces of $C_b(B_X:Y)$:
\begin{align*}
\mathcal{A}_b(B_X:Y) &=\left\{ f\in C_b(B_X:Y): f \mbox{ is holomorphic on the interior of } B_X  \right\}\\ \mathcal{A}_u(B_X:Y) &= \left\{ f\in \mathcal{A}_b(B_X:Y) : f \mbox{ is uniformly continuous on } B_X \right\}.
\end{align*}
We shall denote by $\mathcal{A}(B_X:Y)$ either $\mathcal{A}_b(B_X:Y)$ or $\mathcal{A}_u(B_X:Y)$. In case that $Y$ is the complex filed $\mathbb{C}$, we write $\mathcal{A}(B_X)$, $\mathcal{A}_u(B_X)$ and $\mathcal{A}_b(B_X)$ instead of $\mathcal{A}(B_X:\mathbb{C})$, $\mathcal{A}_u(B_X:\mathbb{C})$ and
$\mathcal{A}_b(B_X:\mathbb{C})$ respectively.

If $X$ and $Y$ are Banach spaces, an \emph{$k$-homogeneous polynomial} $P$ from $X$ to $Y$ is a mapping such that there is an $k$-linear continuous mapping $L$ from $X\times \dots \times X$ to $Y$ such that $$ P(x) = L(x, \dots, x) \quad \text{ for every } x\in X.$$  ${\mathcal P}(^k X:Y)$ denote the Banach space of all $k$-homogeneous polynomials from $X$ to $Y$, endowed with the polynomial norm $\|P\|=\sup_{x \in B_X}{\|P(x)\|}$.  We also say that $P:X\rightarrow Y$ is a \emph{polynomial}, and write $P \in {\mathcal P}(X:Y)$ if $P$ is a finite sum of homogeneous polynomials from $X$ into $Y$. In particular, replace ${\mathcal P}(^k X:Y)$ by $\mathcal{P}({}^kX)$ and ${\mathcal P}(X:Y)$ by $\mathcal{P}(X)$ when $Y$ is a scalar field. We refer to \cite{Din99} for background on polynomials.

The (polynomial) numerical index of a Banach space is a constant relating to the concepts of the numerical radius of functions on $X$. Actually, for each $f \in C_b(B_X:X)$, the \emph{numerical radius} $v(f)$ is defined by $$v(f)=\sup \{|x^*f(x)| : x^*(x)=1, x\in S_X, x^*\in S_{X^*}\},$$ where $X^*$ is the dual space of $X$. For every integer $k\geq 1$, the \emph{$k$-polynomial numerical index} of a Banach space $X$ is the constant defined by \[ n^{(k)}(X) = \inf\{v(P) : \|P\|=1, P\in \mathcal{P}({}^k X:X)\}.\] If $k=1$, in particular, then it is called the numerical index of $X$ and we write $n(X)$. For more recent results about numerical index, see a survey paper \cite{KMP06} and references therein.

Let's briefly see the contents of the paper.
In section 2, using the variational method in \cite{DGZ93}, we show that the set $\rho A$ is a norming subset of $A$ if and only if the set of all strong peak functions in $A$ is a dense $G_\delta$ subset of $A$, when $A$ is a closed subspace $C_b(K:X)$ which contains all elements of the form $ t\mapsto (x^*f(t))^m x$ for all $x\in X$, $x^*\in X^*$, $f\in A$ and integers $m\ge 1$. Using this theorem and the variational methods, we investigate the denseness of the set of strong peak holomorphic functions and the denseness of the set of numerical strong peak functions on certain Banach spaces.

In section 3, we also apply the variation method to investigate the denseness of the set of strongly norm attaining polynomials when the set of all uniformly strongly exposed points of a Banach space $X$ is a norming subset of $\mathcal{P}({}^n X)$. As a direct corollary, the set of all points at which the norm of $\mathcal{P}({}^n X)$ is Fr\'echet differentiable is a dense $G_\delta$ subset if the set of all uniformly strongly exposed points of a Banach space $X$ is a norming subset of $\mathcal{P}({}^n X)$.

In the last part, we will use the graph theory to get some strongly norm-attaining points or complex extreme points. Reisner gave a one-to-one correspondence between n-dimensional some Banach spaces and certain graphs with n vertices. In detail he give a characterization of all finite dimensional real CL-spaces with an absolute norm usig the graph-theoretic terminology. It gives a geometric picture of extreme points of the unit ball of CL-spaces and plays an important role to find the strongly norm-attaining points of $\mathcal{P}({}^k X)$. Moreover we can find all complex extreme points on a complex CL-space with an absolute norm. These strongly norm-attaining points or complex extreme points help answering a problem about the numerical index of a Banach space. We give a partial answer to the Problem 43 in \cite{KMP06}:
\begin{quote}
\emph{Characterize the complex Banach spaces $X$ satisfying $n^{(k)}(X)=1$ for all $k \geq 2$.}
\end{quote}
We show that for a finite dimensional complex Banach space $X$ with an absolute norm, its polynomial numerical indices are one if and only if $X$ is isometric to $\ell_\infty^n$.

For later use, recall the definitions of real and complex extreme points of a unit ball. Let $X$ be a real or complex Banach space. Recall that $x\in B_X$ is said to be an {\it extreme point} of $B_X$ if whenever $y+z=2x$ for some $y,z\in B_X$, we
have $x=y=z$. Denote by $ext(B_X)$ the set of all extreme points of $B_X$. When $X$ is a complex Banach space, an element $x\in B_X$ is said to be a {\it complex
extreme point} of $B_X$ if $\sup_{0\le \theta\le 2\pi} \|x+ e^{i\theta} y\| \le 1$ for some $y\in X$ implies $y=0$. The set of all complex extreme points
 of $B_X$ is denoted by $\ext_\C(B_X)$.


\section{Denseness of the set of strong peak functions}

Let $X$ be a Banach space and $A$ a closed subspace $C_b(K:X)$. A subset $F$ of $K$ is said to be a {\it norming subset} for $A$ if for each $f\in A$, we have
\[ \|f\| = \sup\{ \|f(t)\| :t\in F\}.\] Following the definition of Globevnik in \cite{Glo79}, the smallest closed norming subset of $A$ is called the {\it Shilov boundary} for $A$ and it is shown in \cite{CLS07} that if the set of strong peak functions is dense in $A$ then the Shilov boundary of $A$ exists and it is the closure of $\rho A$. The variation method used in \cite{DGZ93} gives the partial converse of the above mentioned result.

\begin{thm}\label{thm:main}
Let $A$ be a closed subspace $C_b(K:X)$ which contains all elements of the form
$ t\mapsto (x^*f(t))^m x$ for all $x\in X$, $x^*\in X^*$, $f\in A$ and integers $m\ge 1$. Then the set $\rho A$ is a norming subset of $A$
if and only if the set of all strong peak functions in $A$ is a
dense $G_\delta$ subset of $A$.
\end{thm}
\begin{proof}
Let $d$ be the complete metric on $K$. Fix $f\in A$ and
$\epsilon>0$. For each $n\ge 1$, set
\[ U_n =\Big\{ g\in A: \exists z\in \rho A \mbox{ with }
\|(f-g)(z)\| > \sup\{ \|(f-g)(x)\| : d(x,z)> 1/n\}\Big\}.\] Then $U_n$ is
open and dense in $A$. Indeed, fix $h\in A$. Since $\rho A$ is a
norming  subset of $A$, there is a point $w\in \rho A$ such that
\begin{equation}\label{neq1} \|(f-h)(w)\| > \|f-h\| -\epsilon/2. \end{equation}

Choose a peak function $q\in A$ at $w$ with $\|q(w)\| =1$ and $w^*\in S_{X^*}$ such that $w^*q(w)=1$. Then it is easy to see that $w^*\circ q$ is also a strong peak function in $C_b(K)$. So there is an integer $m\ge 1$ such that
$|w^*q(x)|^m<1/3$ for all $x\in K$ with $d(x, w)>1/n$. Now define the function by
\[ p(t) = -(w^*q(t))^m \cdot  \frac{(f-h)(w)}{\|(f-h)(w)\|},\ \ \ \  \ \ \ \forall t\in K.\]

Set $g(x) = h(x) + \epsilon\cdot p(x)$. Then
$\|f(w)-h(w) - \epsilon p(w)\| =\|f(w)-h(w)\| + \epsilon$ and
$\|g-h\|\le \epsilon$. The equation~\eqref{neq1} shows that
\begin{align*}
\|(f-g)(w)\|&=\|f(w)-h(w)-\epsilon p(w)\| = \|f(w)-h(w)\|+ \epsilon\\
&> \|f-h\| + \epsilon/2\\
&\ge \sup\{ \|(f-h)(x) - \epsilon p(x)\|: d(x,w)> 1/n\}.\\
&=\sup\{ \|(f-g)(x)\|: d(x, w)>1/n\}.
\end{align*} Therefore $g\in U_n$.

By the Baire category theorem there is a $g\in \bigcap U_n$ with
$\|g\|<\epsilon$, and we shall show that $f-g$ is a strong peak
function. Indeed, $g\in U_n$ implies that there is $z_n\in X$ such
that
\[ \|(f-g)(z_n)\|> \sup\{ \|(f-g)(x)\|: d(x,z_n)> 1/n\}.\]
Thus $d(z_p, z_n)\le 1/n$ for every $p>n$, and hence $\{z_n\}$
converges to a point $z$, say. Suppose that there is another
sequence $\{x_k\}$ in $B_X$ such that $\{\|(f-g)(x_k)\|\}_k$ converges
to $\|f-g\|$. Then for each $n\ge 1$, there is $M_n\ge 1$ such that
for every $m\ge M_n$,
\[ \|(f-g)(x_m)\|> \sup\{ \|(f-g)(x)\|: d(x,z_n)> 1/n\}.\] Then $d(x_m, z_n)\le
1/n$ for every $m\ge M_n$. Hence $\{x_m\}_m$ converges to $z$. This
shows that $f-g$ is a strong peak function at $z$.

By Proposition~2.19 in \cite{KL07}, the set of all strong peak
functions in $A$ is a $G_\delta$ subset of $A$. This proves the
necessity.

Concerning the converse, it is shown in \cite{CLS07} that if the set of
all strong peak functions is dense in $A$, then the set of all
strong peak points is a norming subset of $A$. This completes the
proof.
\end{proof}

Let $B_X$ be the unit ball of the Banach space $X$. Recall that the
point $x\in B_X$ is said to be a {\it smooth point} if there is a
unique $x^*\in B_{X^*}$ such that $\RE x^*(x)=1$. We denote by ${\rm
sm} (B_X)$ the set of all smooth points of $B_X$. We say that a
Banach space is {\it smooth} if ${\rm sm}(B_X)$ is the unit sphere
$S_X$ of $X$. The following corollary shows that if $\rho A$ is a
norming subset, then the set of smooth points of $B_A$ is dense in
$S_A$.

\begin{cor}\label{cor:main}
Suppose that $K$ is a complete metric space and $A$ is a closed subalgebra of
$C_b(K)$. If $\rho A$ is a norming subset of $A$, then the set of
all smooth points of $B_A$ contains a dense $G_\delta$ subset of
$S_A$.
\end{cor}
\begin{proof}
It is shown in \cite{CLS07} that if $f\in A$ is a strong peak function
and $\|f\|=1$, then $f$ is a smooth point of $B_A$. Then
Theorem~\ref{thm:main} completes the proof.
\end{proof}

Recall  that a Banach space $X$ is said to be a locally uniformly convex if whenever there is a sequence $\{x_n\} \in B_X$  with
 $\lim_n \|x_n + x\|=2$ for some $x\in S_X$, we have $\lim_n \|x_n - x\|=0$.
It is shown in \cite{CLS07} that if $X$ is locally uniformly convex, then
the set of norm attaining elements is dense in $\mathcal{A}(B_X)$.
That is, the set consisting of $f\in \mathcal{A}(B_X)$ with $\|f\|=
|f(x)|$ for some $x\in B_X$ is dense in $\mathcal{A}(B_X)$.

The following corollary gives a stronger result. Notice that if $X$
is locally uniformly convex, then every point of $S_X$ is the strong
peak point for $\mathcal{A}(B_X)$ \cite{Glo79}. Theorem~\ref{thm:main} and
Corollary~\ref{cor:main} implies the following.
\begin{cor}\label{cor:localconvex}
Suppose that $X$ is a  locally uniformly convex, complex Banach
space. Then the set of all strong peak functions in
$\mathcal{A}(B_X)$ is a dense $G_\delta$ subset of
$\mathcal{A}(B_X)$. In particular, the set of all smooth points of
$B_{\mathcal{A}(B_X)}$ contains a dense $G_\delta$ subset of
$S_{\mathcal{A}(B_X)}$.
\end{cor}

It is shown \cite{CHL07} that the set of all strong peak points for
$\mathcal{A}(B_X)$ is dense in $S_X$ if $X$ is an order continuous
locally uniformly $c$-convex sequence space. (For the definition see
\cite{CHL07}). Then by  Theorem~\ref{thm:main}, we get the denseness
of the set of all strong peak functions.

\begin{cor}\label{cor:complexconvex2}
Let $X$ be an order continuous locally uniformly $c$-convex Banach
sequence space. Then the set of all strong peak functions in
$\mathcal{A}_u(B_X:X)$ is a dense $G_\delta$ subset of
$\mathcal{A}_u(B_X:X)$.
\end{cor}

Let $\Pi(X)=\{(x,x^*)\in B_X\times B_{X^*}: x^*(x)=1\}$ be the
topological subspace of $B_X\times B_{X^*}$, where $B_X$ (resp. $B_{X^*}$) is equipped with norm (resp. weak-$*$ compact) topology.

The numerical radius of holomorphic functions was deeply studied in
\cite{Har71} and the denseness of numerical radius holomorphic functions
is studied on the classical Banach spaces \cite{AK07}. Recently the
{\it numerical  strong peak function} is introduced in \cite{KL07} and
the denseness of holomorphic numerical strong peak functions in
$ \mathcal{A}(B_X:X)$ is studied in various Banach spaces. The
function $f\in \mathcal{A}(B_X:X)$ is said to be a numerical strong
peak function if there is $(x,x^*)$ such that $\lim_n
|x^*_nf(x_n)|=v(f)$ for some $\{ ( x_n, x_n^*)\}_n$ in $\Pi(X)$
implies that $(x_n, x_n^*)$ converges to $(x, x^*)$ in $\Pi(X)$. The
function $f\in \mathcal{A}(B_X:X)$ is said to be {\it numerical
radius attaining} if there is $(x, x^*)$ in $\Pi(X)$ such that $v(f)
= |x^*f(x)|$.  An element in the intersection of the set of all strong
peak functions and the set of all numerical strong peak functions is
called a {\it norm and numerical strong peak function} of
$\mathcal{A}(B_X:X)$. Using the variational method again we obtain the following.

\begin{prop}\label{prop:numerical}
Suppose that the set $\Pi(X)$ is complete metrizable and the set
$\Gamma= \{ (x, x^*)\in \Pi(X) : x\in \rho \mathcal{A}(B_X)\cap {\rm
sm}(B_X)\}$ is a numerical boundary. That is, $v(f)
=\sup_{(x,x^*)\in \Gamma} |x^*f(x)|$  for each $f\in
\mathcal{A}(B_X:X)$.  Then the set of all numerical strong peak
functions in $\mathcal{A}(B_X:X)$ is a dense $G_\delta$ subset of
$\mathcal{A}(B_X:X)$.
\end{prop}
\begin{proof}
Let $A=\mathcal{A}(B_X:X)$ and let $d$ be a complete metric on
$\Pi(X)$. In \cite{KL07}, it is shown that if $\Pi(X)$ is complete
metrizable, then the set of all numerical peak functions in $A$ is a
$G_\delta$ subset of $A$. We need prove the denseness. Let $f\in A$
and $\epsilon>0$. Fix $n\ge 1$ and set
\begin{align*} U_n =\{ g\in A&: \exists (z,z^*)\in
\Gamma \mbox{ with }\\
&|z^*(f-g)(z)| > \sup\{ |x^*(f-g)(x)| : d((x,x^*),(z,z^*))> 1/n\}
\}.\end{align*}Then $U_n$ is open and dense in $A$. Indeed, fix
$h\in A$. Since $\Gamma$ is a numerical boundary for $A$, there is a
point $(w, w^*)\in \Gamma$ such that
\[ |w^*(f-h)(w)| > v(f-h) -\epsilon/2.\]
Notice that $d((x,x^*), (w,w^*))>1/n$ implies that there is
$\delta_n>0$ such that $\|x-w\|>\delta_n$. Choose a peak function
$p\in \mathcal{A}(B_X)$ such that $\|p\|=1=|p(w)|$ and $|p(x)|<1/3$
for $\|x-w\|>\delta_n$ and $|w^*(f-h)(w) - \epsilon p(w)|
=|w^*f(w)-w^*h(w)| + \epsilon|p(w)|= |w^*f(w)-w^*h(w)| + \epsilon$.

Put $g(x) = h(x) + \epsilon\cdot p(x)w$. Then
\begin{align*}
|w^*(f-g)(w)|&=|w^*f(w)-w^*h(w)-\epsilon p(w)|= |w^*f(w)-w^*h(w)|+ \epsilon\\
&>v(f-h) + \epsilon/2\\
&\ge \sup\{ |x^*(f-h)(x) - \epsilon p(x)x^*(w)|: d((x,x^*), (w,w^*))> 1/n\}.\\
&=\sup\{ |x^*(f-g)(x)|: d((x,x^*), (w,w^*))>1/n\}.
\end{align*} That is, $g\in U_n$.

By the Baire category theorem there is a $g\in \bigcap U_n$ with
$\|g\|<\epsilon$, and we shall show that $f-g$ is a strong peak
function. Indeed, $g\in U_n$ implies that there is $(z_n, z_n^*)\in
\Gamma$ such that
\[ |z_n^*(f-g)(z_n)|> \sup\{ |x^*(f-g)(x)|: d((x,x^*), (z_n,z_n^*))> 1/n\}.\]
Thus $d((z_p,z_p^*), (z_n,z_n^*))\le 1/n$ for every $p>n$, and hence
$\{(z_n, z_n^*)\}$ converges to a point $(z, z^*)$, say. Suppose
that there is another sequence $\{(x_k, x_k^*)\}$ in $\Pi(X)$ such
that $\{|x_k^*(f-g)(x_k)|\}$ converges to $v(f-g)$. Then for each
$n\ge 1$, there is $M_n\ge 1$ such that for every $m\ge M_n$,
\[ |x_m^*(f-g)(x_m)|> \sup\{ |x^*(f-g)(x)|: d((x,x^*), (z_n, z_n^*))> 1/n\}.\]
Then $d((x_m, x_m^*), (z_n, z_n^*))\le 1/n$ for every $m\ge M_n$.
Hence $\{(x_m, x_m^*)\}$ converges to $(z, z^*)$. This shows that
$f-g$ is a numerical strong peak function at $(z, z^*)$.
\end{proof}

\begin{cor}
Suppose that $X$ is separable, smooth and locally uniformly convex.
Then the set of norm and numerical strong peak functions is a dense
$G_\delta$ subset of $\mathcal{A}_u(B_X:X)$.
\end{cor}
\begin{proof}
It is shown \cite{KL07} that if $X$ is separable, $\Pi(X)$ is complete
metrizable. In view of \cite[Theorem~2.5]{Pal04}, $\Gamma$ is a
numerical boundary for $\mathcal{A}_u(B_X:X)$. Hence
Proposition~\ref{prop:numerical} shows that the set of all numerical
strong peak functions is dense in $\mathcal{A}_u(B_X:X)$. Finally,
Corollary~\ref{cor:localconvex} implies that the set of all norm and
numerical peak functions is a dense $G_\delta$ subset of
$\mathcal{A}_u(B_X:X)$.
\end{proof}

\begin{cor}
Let $X$ be an order continuous locally uniformly $c$-convex, smooth
Banach sequence space. Then the set of all norm and numerical strong
peak functions in $\mathcal{A}_u(B_X:X)$ is a dense $G_\delta$
subset of $\mathcal{A}_u(B_X:X)$.
\end{cor}
\begin{proof}
Notice that $X$ is separable since $X$ is order continuous. Hence
the set of all smooth points of $B_X$ is dense in $S_X$ by the Mazur
theorem and $\Pi(X)$ is complete metrizable \cite{KL07}. In view of
\cite[Theorem~2.5]{Pal04}, $\Gamma$ is a numerical boundary for
$\mathcal{A}_u(B_X:X)$. Hence Proposition~\ref{prop:numerical} shows
that the set of all numerical strong peak functions is a dense
$G_\delta$ subset of $\mathcal{A}_u(B_X:X)$. Theorem~\ref{thm:main}
also shows that the set of all strong peak functions is a dense
$G_\delta$ subset of $\mathcal{A}_u(B_X:X)$. This completes the
proof.
\end{proof}

\section{Denseness of strongly norm attaining polynomials}

Recall that the norm $\|\cdot \|$ of a Banach space is said to be {\it
Fr\'echet differentiable } at $x\in X$ if \[ \lim_{\delta\to 0} \ \sup_{\|y\|=1} \frac{\|x+\delta y\| +
\|x-\delta y\| -2\|x\|}{\delta}=0.\] It is well-known that the set
of Fr\'echet differentiable points in a Banach space is a $G_\delta$
subset \cite[Proposition~4.16]{BL00}. Ferrera \cite{Fer02} shows that in a
real Banach space $X$, the norm of $\mathcal{P}({}^n X)$ is
Fr\'echet differentiable at $Q$ if and only if $Q$ strongly attains
its norm.

 A set $\{x_{\alpha}\}$ of points on the unit sphere $S_X$ of $X$
is called {\it uniformly strongly exposed} (u.s.e.), if there are a
function $\delta(\epsilon)$ with $\delta(\epsilon)>0$ for every
$\epsilon
>0$, and a set $\{f_{\alpha}\}$ of elements of norm $1$ in $X^*$ such that
for every $\alpha$, $ f_{\alpha}(x_{\alpha}) = 1$, and for any $x$,
$$\|x\|\le 1 ~\text{and}~ \RE f_{\alpha}(x) \ge 1- \delta(\epsilon)
~\text{imply}~ \|x-x_{\alpha}\| \le \epsilon.$$

Lindenstrauss \cite[Proposition 1]{Lin63} showed that if $B_X$ is the
closed convex hull of a set of u.s.e. points, then $X$ has property
$A$, that is, for every Banach space $Y$, the set of norm-attaining
elements is dense in $\mathcal{L}(X,Y)$, the Banach space of all
bounded operators from $X$ into $Y$. The following theorem gives
stronger result.

\begin{thm}Let $\mathbb{F}$ be the real or complex scalar field and
$X$, $Y$ Banach spaces over $\mathbb{F}$. For $k\ge 1$, suppose that the u.s.e. points $\{x_\alpha\}$ in $S_X$ is a norming
subset of $\mathcal{P}({}^k X)$. Then the set of strongly norm
attaining elements in $\mathcal{P}({}^k X:Y)$ is dense. In
particular, the set of all points at which the norm of
$\mathcal{P}({}^n X)$ is Fr\'echet differentiable is a dense
$G_\delta$ subset.\end{thm}

\begin{proof}
Let $\{x_\alpha\}$ be a u.s.e. points and $\{x_\alpha^*\}$ be the
corresponding functional which uniformly strongly exposes
$\{x_\alpha\}$. Let $A=\mathcal{P}({}^k X:Y)$, $f\in A$ and
$\epsilon>0$. Fix $n\ge 1$ and set
\[ U_n =\{ g\in A: \exists z\in \rho A \mbox{ with }
\|(f-g)(z)\| > \sup\{ \|(f-g)(x)\| : \inf_{|\lambda|=1}\|x-\lambda
z\|> 1/n\} \}.\] Then $U_n$ is open and dense in $A$. Indeed, fix
$h\in A$. Since $\{x_\alpha\}$ is a norming  subset of $A$, there is
a point $w\in \{x_\alpha\}$ such that
\[ \|(f-h)(w)\| > \|f-h\| -\delta(1/n)\epsilon.\]
Set \[g(x) = h(x) -\epsilon\cdot
p(x)^k\frac{f(w)-h(w)}{\|f(w)-h(w)\|},\] where $p$ is a strongly
exposed functional at $w$ such that $|p(x)|<1-\delta(1/n)$ for
$\inf_{|\lambda|=1}\|x-\lambda w\|>1/n$, $p(w)=1$. Then $\|g-h\|\le
\epsilon$ and
\begin{align*}
\|(f-g)(w)\|&=\left\|f(w)-h(w)+\epsilon p(w)^k\frac{f(w)-h(w)}{\|f(w)-h(w)\|}\right\| = \|f(w)-h(w)\|+ \epsilon\\
&> \|f-h\| + \epsilon(1-\delta(1/n))\\
&\ge \sup\{ \|(f-h)(x) - \epsilon p(x)^k\frac{f(w)-h(w)}{\|f(w)-h(w)\|}\|: \inf_{|\lambda|=1}\|x-\lambda w\|> 1/n\}.\\
&=\sup\{ \|(f-g)(x)|:\inf_{|\lambda|=1}\|x-\lambda w\|> 1/n\}.
\end{align*} That is, $g\in U_n$.

By the Baire category theorem there is a $g\in \bigcap U_n$ with
$\|g\|<\epsilon$, and we shall show that $f-g$ is a strong peak
function. Indeed, $g\in U_n$ implies that there is $z_n\in X$ such
that
\[ \|(f-g)(z_n)\|> \sup\{ \|(f-g)(x)\|: \inf_{|\lambda|=1}\|x-\lambda z_n\|> 1/n\}.\]
Thus $\inf_{|\lambda|=1}\|z_p -\lambda z_n\|\le 1/n$ for every
$p>n$, and $\inf_{|\lambda|=1}|z_p^*(z_n)-\lambda|=
1-|z_n^*(z_p)|\le 1/n$ for every $p>n$. So $\lim_n\inf_{p>n}
|z_n^*(z_p)|=1$. Hence there is a subsequence of $\{z_n\}$ which
converges to $z$, say by \cite[Lemma~6]{Aco91b}. Suppose that there is
another sequence $\{x_k\}$ in $B_X$ such that $\{\|(f-g)(x_k)\|\}$
converges to $\|f-g\|$. Then for each $n\ge 1$, there is $M_n$ such
that $M_n\ge n$ and  for every $m\ge M_n$,
\[ \|(f-g)(x_m)\|> \sup\{ \|(f-g)(x)\|:\inf_{|\lambda|=1} \|x-\lambda z_n\|> 1/n\}.\] Then
$\inf_{|\lambda|=1}\|x_m-\lambda z_n\|\le 1/n$ for every $m\ge M_n$.
So $\inf_{|\lambda|=1}\|x_m-\lambda z\|\le
\inf_{|\lambda|=1}\|x_m-\lambda z_n \| + \|z- z_n \|\le 2/n$ for
every $m\ge M_n$. Hence we get a convergent subsequence of $x_n$ of
which limit is $\lambda z$ for some $\lambda\in S_\mathbb{C}$. This
shows that $f-g$ strongly norm attains at $z$.

It is shown in \cite{CLS07} that the norm is Fr\'echet differentiable at
 $P$ if and only if whenever there are sequences $\{t_n\}$,
$\{s_n\}$ in $B_X$ and scalars $\alpha$, $\beta$ in $S_\mathbb{F}$
such that $\lim_n \alpha P(t_n)=\lim_n \beta P(s_n) = \|P\|$, we get
\begin{equation}\label{eq1}\lim_n\ \sup_{\|Q\|=1}|\alpha Q(t_n)-\beta Q(s_n)|=0.\end{equation}
We have only to show that every nonzero element $P$ in $A$ which
strongly attains its norm satisfies the condition~(\ref{eq1}).
Suppose that $P$ strongly attains its norm at $z$ and $P\neq 0$.

For each $Q\in A$, there is a $k$-linear form $L$ such that $Q(x) =
L(x,\dots,x)$ for each $x\in X$. The polarization identity \cite{Din99}
shows that $\|Q\|\le \|L\|\le (k^k/k!)\|Q\|$. Then for each $x,y\in
B_X$, $\|Q(x)- Q(y)\|\le n\|L\|\|x-y\|$ and
\[\|Q(x)- Q(y)\|\le \frac{k^{k+1}}{k!} \|Q\|\|x-y\|.\]

Suppose that there are sequences $\{t_n\}$, $\{s_n\}$ in $B_X$ and
scalars $\alpha$, $\beta$ in $S_\mathbb{F}$ such that $\lim_n \alpha
P(t_n)=\lim_n \beta P(s_n) = \|P\|$, then for any subsequences
$\{s_n'\}$ of $\{s_n\}$ and $\{t_n'\}$ of $\{t_n\}$, there are
convergent further subsequences $\{t''_n\}$ of $\{t_n'\}$ and
$\{s_n''\}$ of $\{s_n'\}$ and scalars $\alpha''$ and $\beta''$ in
$S_\mathbb{F}$ such that $\lim_n t''_n = \alpha'' z$ and $\lim_n
s_n'' = \beta'' z$. Then $\alpha P(\alpha''z) = \beta P(\beta''z) =
1$. So $\alpha(\alpha'')^k = \beta(\beta'')^k$.

Then we get
\begin{align*}\overline{\lim_n}\ \sup_{\|Q\|=1} |\alpha Q(t''_n) - \beta
Q(s''_n)| &\le \overline{\lim_n}\ \sup_{\|Q\|=1} \left[|\alpha
Q(t''_n) - \alpha
Q(\alpha''z)| + |\beta Q(\beta'' z) - \beta Q(s''_n)|\right]\\
&\le \overline{\lim_n}  \frac{k^{k+1}}{k!} (\|t''_n - \alpha''z\| +
\| \beta'' z-  s''_n\|)=0.
\end{align*}
This implies that $\lim_n \sup_{\|Q\|=1}|\alpha Q(t_n)- \beta
Q(s_n)|=0$. Therefore the norm is Fr\'echet differentiable at $P$.
This completes the proof.
\end{proof}

\begin{rem}
Suppose that the $B_X$ is the closed convex hull of a set of  u.s.e
points, then the set of u.s.e. points is a norming subset of
$X^*=\mathcal{P}({}^1 X)$. Hence the elements in $X^*$ at which the
norm of $X^*$ is Fr\'echet differentiable is a dense $G_\delta$
subset.
\end{rem}


\section{Polynomial numerical index and Graph theory}

A norm $\|\cdot\|$ on $\R^n$ or $\C^n$ is said to be an \emph{absolute norm} if $\|(a_1,\ldots,a_n)\| = \|(|a_1|,\ldots,|a_n|)\|$ for every scalar $a_1,\ldots,a_n$, and $\|(1,0,\ldots,0)\| = \cdots = \|(0,\ldots,0,1)\| = 1$.
We may use the fact that the absolute norm is less than or equal to the $\ell_1$-norm and it is nondecreasing in each variable.

A real or complex Banach space $X$ is said to be a \emph{CL-space} if its unit ball is the absolutely convex hull of every maximal convex subset of the unit sphere. In particular, if $X$ is finite dimensional, then it is equivalent to the condition that $|x^*(x)|=1$ for every $x^* \in \ext B_{X^*}$ and every $x \in \ext B_X$ \cite{Lim78}.

Let $X$ be a n-dimensional Banach space with an absolute norm $\|\cdot\|$. In this section, $X$ as a vector space is considered $\R^n$ or $\C^n$ and we denote by $\{e_j\}_{j=1}^n$ and $\{e_j^\ast\}_{j=1}^n$ the canonical basis and the coordinate functionals, respectively. We also denote by $\ext B_X$ the set of all extreme points of $B_X$.

Now define the following mapping between n-dimensional Banach spaces with an absolute norm and certain graphs with n vertices :
$$
X \quad \longmapsto \quad G=G(X),
$$
where $G$ is a graph with the vertex set $V=\{1,2,\ldots,n\}$ and the edge set $E=\{(i,j) : e_i+e_j\notin B_X\}$. For example, if $X=\ell_1^n$, then $G(X)$ is a complete graph of order $n$, that is, a graph in which every pair of $n$ vertices is connected by an edge. Conversely, $G(X)$ is a graph in which no pairs of $n$ vertices are connected by an edge if $X=\ell_\infty^n$. Using these graphs and their theory, Reisner gave the exact characterization of all finite dimensional CL-spaces with an absolute norm. Prior to Reisner's theorem, we give some basic definitions in the graph theory.

Given a graph $G=(V,E)$, we say that $\sigma \subset V$ is a \emph{clique} of $G$ if the edge set $E$ of $G$ contains all pairs consisting of any two vertices in $\sigma$. Conversely, $\tau \subset V$ is called a \emph{stable set} of $G$ if $E$ contains no pairs consisting of two vertices in $\tau$. A graph $G$ is said to be \emph{perfect} if $$\omega(H) = \chi(H) \text{ for every induced subgraph } H \text{ of } G,$$
where $\omega(G)$ denotes the clique number of $G$ (the largest cardinality of a clique of $G$) and $\chi(G)$ is the chromatic number of $G$ (the smallest number of colors  needed to color the vertices of $G$ so that no two adjacent vertices share the same color).

\begin{thm}\cite[Reisner]{Rei91}\label{graph}
Let $X$ be a finite dimensional Banach space with an absolute norm. Then $X$ is a CL-space if and only if $G=G(X)$ is a perfect graph and there exists a unique common element between every maximal clique and each maximal stable set of $G$.
In particular, for every n-dimensional CL-space $X$, the following characterizations of the set of extreme points of $B_X$ and $B_{X^\ast}$ hold respectively:
\begin{enumerate}
\item
$(x_1,x_2,\ldots,x_n)\in \ext B_X$ if and only if $|x_j| \in \{0,1\}$ for all $j=1,2,\ldots,n$ and the support $\{j\in V: x_j \neq 0\}$ is a maximal stable set of $G$.
\item
$(x_1,x_2,\ldots,x_n)\in \ext B_{X^\ast}$ if and only if $|x_j| \in \{0,1\}$ for all $j=1,2,\ldots,n$ and the support $\{j\in V: x_j \neq 0\}$ is a maximal clique of $G$.
\end{enumerate}
\end{thm}

In this theorem, the maximality of cliques and stable sets comes from the partial order of inclusion.

Let $X$ be a finite dimensional CL-space. If $\tau$ is a maximal clique of $G=G(X)$, then a subspace $\vspan \{e_j : j\in \tau\}$ of $X$ is isometrically isomorphic to $\ell_1^{|\tau|}$. Indeed, since the absolute norm is less than or equal to the $\ell_1$-norm, we have $\|\sum_{j\in \tau} a_j e_j\| \leq \sum_{j\in \tau} |a_j|$ for every scalar $a_j$. For the inverse inequality, let $x^*_\tau = \sum_{j\in \tau} \overline{\sign} (a_j) \cdot e_j^*$. Then $x^*_\tau$ is in $\ext B_{X^*}$ by Theorem~\ref{graph} and hence $x^*_\tau(\sum_{j\in \tau} a_j e_j) = \sum_{j\in \tau} |a_j|$, which completes the proof.

\begin{rem}\label{rem:graph}
Originally, Reisner just proved the above theorem for the real case. However, it can be extended to the general case (real or complex). For this, we need the following comments and proposition.
\end{rem}

There is a natural one-to-one correspondence between the absolute norm of $\R^n$ and the one of $\C^n$. Specifically, given real Banach space $X=(\R^n,\|\cdot\|)$ with an absolute norm, we can find the complexification $\tilde{X}=(\C^n,\|\cdot\|_{\C})$ of $X$ defined by $\|(z_1,\ldots,z_n)\|_{\C} := \|(|z_1|,\ldots,|z_n|)\|$ for each $(z_1,\ldots,z_n) \in \tilde{X}$. Then $\tilde{X}$ is clearly the complex Banach space with the absolute norm. Moreover we get the following basic proposition.

\begin{prop}\label{propbasic}
Let $X$ be a real Banach space with an absolute norm and $\tilde X$ the complexification of $X$. Then, for an element $x$ of $X$, $x\in \ext B_X$ if and only if $x\in \ext B_{\tilde X}$. In particular, $X$ is a CL-space if and only if $\tilde X$ is a CL-space.
\end{prop}
\begin{proof}
The sufficiency is clear. Suppose that $x$ is an extreme point of $B_X$ and $2x = y + z$ for some $y, z$ in $B_{\tilde X}$. We claim that $y$ and $z$ are in $X$. For the contrary, suppose that some $j^{th}$-coordinate of $y$ is not a real number. That is, $e_j^*(y)= a+bi$, $a,b\in\R$ and $b\neq 0$. Since $2x = y + z = \RE(y) +\RE(z) + i(\IM(y)+ \IM(z))$, we have $\RE(y)=\RE(z)=x$, where $\RE(x) =\sum_{k=1}^n \RE(e_k^*(x)) e_k$ and $\IM(x) =\sum_{k=1}^n \IM(e_k^*(x)) e_k$. Take a positive real number $\delta$ less than $\sqrt{a^2+b^2} -|a|$. Note that $\|x\pm \delta e_j\| \leq \|y\| \leq 1$ by the basic property of an absolute norm. So $2x = (x+\delta e_j) + (x-\delta e_j)$ contradicts to the fact that $x$ is an extreme point of $B_X$.
\end{proof}

Using the graph-theoretic technique on CL-spaces, we are about to find the strongly norm attaining points of $\tilde{\rho}\mathcal{P}({}^m X)$. For this, let us consider the following lemma.
\begin{lem}\label{lem:str_att_poly}
Let $Y=\ell^N_1$ and let $m$ be a positive integer. For any $j_1,j_2,\ldots,j_m \in \{1,2,\ldots,N\}$, define a m-homogeneous polynomial $Q_{j_1,j_2,\ldots,j_m}$ of $Y$ by $$Q_{j_1,j_2,\ldots,j_m}(x) = \prod^{m}_{k=1}e^\ast_{j_k}(x) + \Big[\sum_{j\in \{j_1, \ldots, j_m\}} e^\ast_{j}(x)\Big]^m.$$ Then $Q_{j_1,j_2,\ldots,j_m}$ attains its norm only at $\frac{c}{m}\sum^{m}_{k=1}e_{j_k}$, $|c|=1$. Hence $Q_{j_1,j_2,\ldots,j_m}$ strongly attains its norm at $\frac{1}{m}\sum^{m}_{k=1}e_{j_k}$.
\end{lem}

\begin{proof} For positive integers $m_1, \ldots, m_n$, consider the product $$(x_1,\ldots,x_n) \longmapsto \prod_{k=1}^n x_k^{m_k}$$  on the compact subset $\R_+^n\cap S_{\ell_1^n}$. Then it is easy to see by induction that the product has the unique maximum at $(\frac{m_1}{m_1 + \cdots + m_n}, \ldots, \frac{m_n}{m_1 + \cdots + m_n})$. Hence the norm of the polynomial $\prod^{m}_{k=1}e^\ast_{j_k}(x)$ is attained only at $x=(x_1, \ldots, x_N)$, where $|x_j| = \frac 1m \sum_{k=1}^m e^*_j (e_{j_k})$ for each $1\le j \le N$. Notice also that the norm of the polynomial $(\sum^{N}_{j=1}e^\ast_{n}(x))^m$ is attained only at $x=(x_1, \ldots, x_N)$, where $\sign(x_1)= \dots = \sign(x_N)$ and $x\in S_{\ell_1^N}$. Hence $Q_{j_1,j_2,\ldots,j_m}$ attains its norm only at $\frac cm \sum_{k=1}^m e_{j_k}$ for some $c\in S_\C$. This completes the proof.
\end{proof}

\begin{thm}\label{poly_peak_pt}
Let $X$ be a finite dimensional (real or complex) CL-space with an absolute norm and let $m$ be a positive integer. Then $\frac{1}{m}\sum^{m}_{j=1}y_j \in \tilde{\rho}\mathcal{P}({}^m X)$ whenever $y_1,y_2,\ldots,y_m$ are extreme points of $B_X$ whose coordinates are nonnegative real numbers.
\end{thm}

\begin{proof}
Denote by $M(G)$ the family of all maximal cliques of $G=G(X)$ and let $y_1,y_2,\ldots,y_m$ be extreme points of $B_X$ whose coordinates are nonnegative real numbers.
For each $J\in M(G)$, define the $m$-homogeneous polynomial $Q_J$ and linear functional $L_J$ as the following
$$
Q_J = Q_{j_1,j_2,\ldots,j_m}, \quad   L_J = \sum_{j\in \{j_1, \ldots, j_m\}} e^*_j,
$$
where each $j_k$ ($k=1,2,\ldots,m$) is a unique common element between a maximal clique $J$ and the support of an extreme point $y_k$.
Now define a $m$-homogeneous polynomial $Q$ of $X$ by
$$
Q = \sum_{J\in M(G)} Q_J + \Big[\sum_{J\in M(G)} L_J\Big]^m.
$$

For every maximal clique $J$ of $G$, denote by $P_J$ the projection of $X$ onto $\vspan \{e_j : j\in J \}$. Then, it is clear that
$$
P_J \Big(\frac{1}{m}\sum^{m}_{j=1}y_j\Big) = \frac{1}{m}\sum^{m}_{j=1}P_J (y_j) = \frac{1}{m}\sum^{m}_{k=1}e_{j_k}.
$$ Notice that $Q_J \circ P_J = Q_J$ and $L_J \circ P_J = L_J$ for each $J\in M(G)$.
It follows by Theorem~\ref{complex_ext} that
\begin{align*}
&Q\Big(\frac{1}{m}\sum^{m}_{j=1}y_j\Big) \\
&= \sum_{J\in M(G)} Q_J\Big(\frac{1}{m}\sum^{m}_{j=1}y_j\Big) + \Big[\sum_{J\in M(G)} L_J\Big(\frac{1}{m}\sum^{m}_{j=1}y_j\Big)\Big]^m \\
&= \sum_{J\in M(G)} Q_J\Big(P_J\big(\frac{1}{m}\sum^{m}_{j=1}y_j\big)\Big) + \Big[\sum_{J\in M(G)} L_J\Big(P_J\big(\frac{1}{m}\sum^{m}_{j=1}y_j\big)\Big)\Big]^m\\
&= \sum_{J\in M(G)} Q_J\Big(\frac{1}{m}\sum^{m}_{k=1}e_{j_k}\Big) + \Big[\sum_{J\in M(G)} L_J\Big(\frac{1}{m}\sum^{m}_{k=1}e_{j_k}\Big)\Big]^m \\
&= \sum_{J\in M(G)} \|Q_J\| + \Big[\sum_{J\in M(G)} \|L_J\|\Big]^m\\
&= \quad \big\|Q\big\|,
\end{align*}
and that $Q$ attains its norm at $\frac{1}{m}\sum^{m}_{j=1}y_j$.

Then we claim that the above polynomial $Q$ strongly attains its norm at $\frac{1}{m}\sum^{m}_{j=1}y_j$. Now, suppose that $|Q(y)|  = \|Q\|$ and we need to show that $y=\frac cm \sum_{j=1}^m y_j$ for some $c\in S_\C$. Then we have the following inequalities.
\begin{align*}
|Q(y)| &= \Big|\sum_{J\in M(G)} Q_J(P_J y)\Big| + \Big|\sum_{J\in M(G)} L_J(P_J y) \Big|^m \\
&\leq \sum_{J\in M(G)} |Q_J(P_J y)| + \Big[\sum_{J\in M(G)} | L_J (P_J y ) |\Big]^m\\
&\leq \sum_{J\in M(G)} \|Q_J\|  + \Big[\sum_{J\in M(G)} \| L_J \|\Big]^m\\
&= \quad \big\|Q\big\|.
\end{align*}
Hence
$ |Q_J(P_J y)| = \|Q_J\|$, $L_J(P_J y) = \|L_J\|$ for each $J\in M(G)$ and $\sign L_J(P_J y)$ are all the same for all $J\in M(G)$. By  Theorem~\ref{complex_ext}, $P_J y =\frac{c_J}{m}\sum^{m}_{k=1}e_{j_k}$  for each $J\in M(G)$. Since $\sign(L_J(P_J y)) = c_J$ and they are all the same for all $J\in M(G)$,  take $c_J = c$ for all $ J\in M(G)$.

Since each maximal clique $J$ induces every extreme point $x_J^\ast$ of $B_{X^\ast}$, the above equation implies that
$$
x_J^*(y) = x_J^\ast(P_J y)= x_J^*P_J\Big(\frac{c}{m}\sum^{m}_{j=1}y_j\Big) = x_J^*\Big(\frac{c}{m}\sum^{m}_{j=1}y_j\Big)
$$
Consequently $y= \frac{c}{m}\sum^{m}_{j=1}y_j$. This completes the proof.
\end{proof}

It was shown in \cite{Lee07b} that a necessary and sufficient condition for a Banach space with the Radon-Nikodym Property to have the polynomial numerical index one can be stated as follows:
$\quad n^{(k)}(X)=1$ if and only if
\begin{equation}\label{lee_polyindex}
|x^{\ast}(x)|=1 \textrm{ for all }x \in \tilde{\rho}\mathcal{P}({}^k X) \textrm{ and } x^\ast \in \ext B_{X^\ast}.
\end{equation}

The following theorem is a partial answer to Problem 43 in \cite{KMP06} :
\begin{quote}
\emph{Characterize the complex Banach spaces $X$ satisfying $n^{(k)}(X)=1$ for all $k \geq 2$.}
\end{quote}

\begin{thm}\label{polyindex}
Let $X$ be a n-dimensional complex Banach space with an absolute norm and let $k$ be an integer greater than or equal to 2. Then $n^{(k)}(X)=1$ if and only if $X$ is isometric to $\ell_\infty^n$.
\end{thm}
\begin{proof}
Suppose that $n^{(k)}(X)=1$. Since $n^{(k)}(X)=1$ implies $n(X)=1$ (i.e. $X$ is a CL-space), we can apply Theorem \ref{graph}. To show that $X$ is isometric to $\ell_\infty^n$, it suffices to prove that $B_X$ has only one extreme point whose coordinates are nonnegative real number. Suppose that there exist distinct two extreme points $x$, $y$ of $B_X$ whose coordinates are all nonnegative. Take a maximal clique $\tau$ of $G=G(X)$. Then, by Theorem \ref{graph}, there exists a unique common element $i$ between a maximal clique $\tau$ and the support of $x$. Similarly, let $j$ be a common element between $\tau$ and the support of $y$. Now consider an $x^\ast_\tau \in \ext B_{X^\ast}$ defined by
$$
x^\ast_\tau (e_k) =
\begin{cases}
-1 & \text{, if } k=i \\
1 & \text{, if } k\in \tau \backslash \{i\} \\
0 & \text{, if } k\notin \tau .
\end{cases}
$$
Then, using Theorem \ref{graph}, we get
$$ x^\ast_\tau (\frac{x+y}{2})=\frac{x^\ast_\tau (x)+x^\ast_\tau (y)}{2}=\frac{-1+1}{2}=0. $$
However, since $\frac{x+y}{2} \in \tilde{\rho}\mathcal{P}({}^2 X)$ by Theorem \ref{poly_peak_pt}, this contradicts to (\ref{lee_polyindex}).

For the converse, note that $\tilde{\rho}\mathcal{P}({}^k X) \subset \ext_{\C} B_X$ by \cite[Proposition 2.1]{Lee07b}. It is also easy to see that $\ext_{\C} B_X = \ext B_X$ when $X=\ell_\infty^n$. After all, it follows from (\ref{lee_polyindex}) that $n^{(k)}(\ell^n_\infty)=1$.
\end{proof}

\section{Characterization of complex extreme points}

\begin{thm}\label{complex_ext}
Let $X = (\C^n,\|\cdot\|)$ be a n-dimensional complex CL-space with an absolute norm. Then an element $(a_1,a_2,\ldots,a_n)$ in $X$ is an complex extreme point of $B_X$ if and only if $(|a_1|,|a_2|,\ldots,|a_n|)$ is a convex combination of extreme points of $B_X$ whose coordinates all are positive real numbers. In short, $\R_{+}^n \cap \ext_{\C} B_X = \co (\R_{+}^n \cap \ext B_X)$.
\end{thm}

\begin{proof}
Let $(a_1,a_2,\ldots,a_n)$ be an complex extreme point of $B_X$. If $\tilde{X}=(\R^n,\|\cdot\|)$, then all extreme points of $B_{\tilde{X}}$ are extreme points of $B_X$ by Proposition \ref{propbasic}. So an element $x:=(|a_1|,|a_2|,\ldots,|a_n|)$ has an expression $x=\sum^{m}_{j=1}\lambda_j y_j$ where $\sum^{m}_{j=1}\lambda_j =1$, $\lambda_j >0$ and each $y_j$ is an extreme point of $B_X$ whose coordinates are real numbers. Now we claim that all coordinates of each $y_j$ are positive. For the contradiction, suppose that the first coordinate of some $y_j$ is negative. More specifically, assume that the first coordinate of $y_j$ is $-1$ if $1\leq j\leq r$ and 1 otherwise. Then $|a_1|<1$. Note that for $y'_j=y_j + 2e_1$, $j=1,2,\ldots,r$,
\begin{eqnarray*}
(1,|a_2|,\ldots,|a_n|) &=& \sum^{r}_{j=1}\lambda_j y'_j + \sum^{m}_{j=r+1}\lambda_j y_j\\
&\in& \co (\ext B_X) = B_X .
\end{eqnarray*}
It follows that for all $\eps \in \C$ with $|\eps|< 1-|a_1|$,
\begin{eqnarray*}
\|x + \eps e_1\| &=& \big\|(|a_1|+\eps,|a_2|,\ldots,|a_n|)\big\| \\
&\leq& \big\|(1,|a_2|,\ldots,|a_n|)\big\| \quad \leq  \quad 1.
\end{eqnarray*}
So $(|a_1|,|a_2|,\ldots,|a_n|)$ is not a complex extreme point, which is a contradiction.

For the converse, let $x=\sum^{m}_{j=1}\lambda_j y_j$ where $\sum^{m}_{j=1}\lambda_j =1$ and $y_j \in \ext B_X$ for all $j=1,2,\ldots,m$. Suppose that $x$ is not a complex extreme point of $B_X$. Then there exists a nonzero $y$ in $X$ such that $x + \eps y \in B_X$ whenever $|\eps|\leq 1$. Take a maximal clique $\tau$ in $V$ containing some vertices which are nonzero coordinates of $y$. Consider the projection $P_\tau$ of $X$ onto the linear span $Y$ of $\{e_j : j\in \tau\}$. Then it follows from assumption that $P_\tau x$ is also not a complex extreme point of $B_Y$. Moreover, we can check that $P_\tau x$ is on the unit sphere of $X$. Indeed, if an extreme point $x^\ast_\tau$ of $B_{X^\ast}$ is defined by $x^\ast_\tau(e_j)= 1$ for $j\in \tau$ and $x^\ast_\tau(e_j)=0$ for $j\notin \tau$, then we have $x^\ast_\tau(y_j)=1$ for all $j=1,2,\ldots,m$ by Theorem \ref{graph}. Consequently,
$$
x^\ast_\tau(P_\tau x) = x^\ast_\tau(x) = \sum^{m}_{j=1}\lambda_j x^\ast_\tau(y_j) = \sum^{m}_{j=1}\lambda_j = 1
.$$
Now, from the fact that $Y = \textrm{span} \{e_j : j\in\tau\}$ is isometrically isomorphic to $\ell_1^{|\tau|}$ since $\tau$ is a clique, we get that every element of norm one in $Y$ is a complex extreme point of $B_Y$ \cite{Glo75}. This is a contradiction.
\end{proof}

The above characterization of complex extreme points have some application in the theory about the numerical index of Banach spaces. Specifically we want to apply to the \emph{analytic numerical index} of $X$ defined by
$$ n_a(X) = \inf\{v(P) : \|P\|=1, P\in \mathcal{P}(X:X)\} $$
It is easily observed that $0 \leq n_a(X) \leq n^{(k)}(X) \leq n(X) \leq 1$ for every $k\geq 2$. Note that a necessary and sufficient condition for a finite dimensional complex Banach space to have the analytic numerical index 1 can be stated using complex extreme points as follows (see \cite{Lee07b}.): $\quad n_a(X)=1$ if and only if
\begin{equation}\label{lee_holoindex}
|x^{\ast}(x)|=1 \textrm{ for all }x \in \ext_{\C} B_X \textrm{ and } x^\ast \in \ext B_{X^\ast}.
\end{equation}

After applying the characterization of complex extreme points to the above theorem, we have the following corollary by the same argument as the proof of Theorem~\ref{polyindex}.

\begin{cor}\label{holoindex}
Let $X$ be an n-dimensional complex Banach space with an absolute norm. Then $n_a(X)=1$ if and only if $X$ is isometric to $\ell_\infty^n$.
\end{cor}

As an immediate consequence of the above theorem, we get a corollary. For further details, we need some definitions about the Daugavet property. A function $\Phi \in \ell_\infty(B_X,X)$ is said to satisfy the \emph{(resp. alternative) Daugavet equation} if the norm equality
$$\|Id + \Phi\|= 1 + \|\Phi\|$$
$$ \left(\text{ resp. } \sup_{\omega \in S_\C} \|Id + \omega \Phi\|= 1 + \|\Phi\|\ \right)$$

holds. If this happens for all weakly compact polynomials in $\mathcal{P}(X:X)$, we say that $X$ has the \emph{(resp. alternative) p-Daugavet property}. Similarly, $X$ is said to have the \emph{$k$-order Daugavet property} if the Daugavet equation is satisfied for all rank-one $k$-homogeneous polynomials in $\mathcal{P}(^kX:X)$.

\begin{cor}\label{dp_index}
Let $X$ be a finite dimensional complex Banach space with an absolute norm. Then the followings are equivalent.
\begin{enumerate}
\item[$(a\,)$] $X$ has the alternative p-Daugavet property.
\item[$(b\,)$] $X$ has the $k$-order Daugavet property for some $k\geq 2$.
\item[$(b')$] $X$ has the $k$-order Daugavet property for every $k\geq 2$.
\item[$(c\,)$] $n^{(k)}(X)=1$ for some $k\geq 2$.
\item[$(c')$] $n^{(k)}(X)=1$ for every $k\geq 2$.
\item[$(d\,)$] $n_a(X)=1$
\item[$(e\,)$] $X = \ell^n_\infty$
\end{enumerate}
\end{cor}

\begin{proof}
Theorem~\ref{polyindex} and Theorem~\ref{holoindex} imply $(c)\Leftrightarrow (e)$ and $(d)\Leftrightarrow (e)$ respectively. $(b) \Rightarrow (c)$ or $(b') \Rightarrow (c')$ is induced by \cite[Proposition~1.3]{CGMM07}.
It is also easy to check that
$$
\begin{array}{ccccccc}
& & (b) & \Rightarrow & (c) & \Leftrightarrow & (e) \\
& & \Uparrow & & \Uparrow & & \Updownarrow \\
(a) & \Rightarrow & (b') & \Rightarrow & (c') & \Leftarrow & (d) \\
\end{array}
$$
Corollary 2.10 in \cite{CGMM07} shows $(e)\Rightarrow (a)$. The proof is done.
\end{proof}

\begin{rem}
Let $X$ be a finite dimensional (real or complex) Banach space with an absolute norm. Then it is impossible that $X$ has the p-Daugavet property. Indeed, if $X$ is a real (resp. complex) Banach space, then $n^{(k)}(X)=1$ for all $k\ge 2$ and $X=\R$ (resp. $X=\ell_\infty^n$ for some $n$) (see \cite{Lee07b} and Corollary~\ref{polyindex}). Then it is easy to see that both $\R$ and complex $\ell_\infty^n$ do not have $p$-Daugavet property.
\end{rem}

It is worth to note that the complex extreme points play an important role in the study of norming subsets of $\mathcal{A}(B_X)$. In particular, $\rho \mathcal{A}(B_X)=\ext_\C B_X$ when $X$ is a finite dimensional complex Banach space (see \cite{CHL07}). Using the result in \cite{HN05}, we obtain the following corollary. For the further references about complex convexity and monotonicity, see \cite{LeeHJ05, LeeHJ07, Lee07}.

\begin{cor}\label{real_mon}
Let $X = (\R^n,\|\cdot\|)$ be an n-dimensional real CL-space with an absolute norm. Then an element $(|a_1|,|a_2|,\ldots,|a_n|)$ in $S_X$ is a point of upper monotonicity of $X$ if and only if $(|a_1|,|a_2|,\ldots,|a_n|)$ is a convex combination of extreme points of $B_X$ whose coordinates all are positive real numbers.
\end{cor}


\bibliographystyle{amsalpha}

\begin{thebibliography}{CGMM07}

\bibitem[Aco91]{Aco91b}
Mar{\'{\i}}a~D. Acosta, \emph{Denseness of numerical radius attaining
  operators: renorming and embedding results}, Indiana Univ. Math. J.
  \textbf{40} (1991), no.~3, 903--914.

\bibitem[AK07]{AK07}
Mar{\'{\i}}a~D. Acosta and Sung~Guen Kim, \emph{Denseness of holomorphic
  functions attaining their numerical radii}, Israel J. Math. \textbf{161}
  (2007), 373--386.

\bibitem[BL00]{BL00}
Yoav Benyamini and Joram Lindenstrauss, \emph{Geometric nonlinear functional
  analysis. {V}ol. 1}, American Mathematical Society Colloquium Publications,
  vol.~48, American Mathematical Society, Providence, RI, 2000.

\bibitem[CGMM07]{CGMM07}
Yun~Sung Choi, Domingo Garc{\'{\i}}a, Manuel Maestre, and Miguel Mart{\'{\i}}n,
  \emph{The {D}augavet equation for polynomials}, Studia Math. \textbf{178}
  (2007), no.~1, 63--82.

\bibitem[CHL07]{CHL07}
Yun~Sung Choi, Kwang~Hee Han, and Han~Ju Lee, \emph{Boundaries for algebras of
  holomorphic functions on {B}anach spaces}, Illinois J. Math. \textbf{51}
  (2007), no.~3, 883--896.

\bibitem[CLS]{CLS07}
Yun~Sung Choi, Han~Ju Lee, and Hyun~Gwi Song, \emph{Bishop's theorem and
  differentiability of a subspace of ${C}_b({K})$}, preprint (2007), \url{http://arxiv.org/abs/0708.4069}.

\bibitem[DGZ93]{DGZ93}
Robert Deville, Gilles Godefroy, and V{\'a}clav Zizler, \emph{A smooth
  variational principle with applications to {H}amilton-{J}acobi equations in
  infinite dimensions}, J. Funct. Anal. \textbf{111} (1993), no.~1, 197--212.

\bibitem[Din99]{Din99}
Se{\'a}n Dineen, \emph{Complex analysis on infinite-dimensional spaces},
  Springer Monographs in Mathematics, Springer-Verlag London Ltd., London,
  1999.

\bibitem[Fer02]{Fer02}
Juan Ferrera, \emph{Norm-attaining polynomials and differentiability}, Studia
  Math. \textbf{151} (2002), no.~1, 1--21.

\bibitem[Glo75]{Glo75}
Josip Globevnik, \emph{On complex strict and uniform convexity}, Proc. Amer. Math.
  Soc. \textbf{47} (1975), 175--178.

\bibitem[Glo79]{Glo79}
Josip Globevnik, \emph{Boundaries for polydisc algebras in infinite
  dimensions}, Math. Proc. Cambridge Philos. Soc. \textbf{85} (1979), no.~2,
  291--303.

\bibitem[Har71]{Har71}
Lawrence~A. Harris, \emph{The numerical range of holomorphic functions in
  {B}anach spaces}, Amer. J. Math. \textbf{93} (1971), 1005--1019.

\bibitem[HN05]{HN05}
Henryk Hudzik and Agata Narloch, \emph{Relationships between monotonicity and
  complex rotundity properties with some consequences}, Math. Scand.
  \textbf{96} (2005), no.~2, 289--306.

\bibitem[KL]{KL07}
Sung~Guen Kim and Han~Ju Lee, \emph{Norm and numerical peak holomorphic
  functions on banach spaces}, preprint, (2007), \url{http://arxiv.org/abs/0706.0574}.

\bibitem[KMP06]{KMP06}
Vladimir Kadets, Miguel Mart{\'{\i}}n, and Rafael Pay{\'a}, \emph{Recent
  progress and open questions on the numerical index of {B}anach spaces},
  RACSAM Rev. R. Acad. Cienc. Exactas F\'\i s. Nat. Ser. A Mat. \textbf{100}
  (2006), no.~1-2, 155--182.

\bibitem[Lee05]{LeeHJ05}
Han~Ju Lee, \emph{Monotonicity and complex convexity in {B}anach lattices}, J.
  Math. Anal. Appl. \textbf{307} (2005), no.~1, 86--101.

\bibitem[Lee07a]{LeeHJ07}
\bysame, \emph{Complex convexity and monotonicity in quasi-{B}anach lattices},
  Israel J. Math. \textbf{159} (2007), 57--91.

\bibitem[Lee07b]{Lee07}
\bysame, \emph{Randomized series and geometry of {B}anach spaces}, preprint
  (2007), \url{http://arxiv.org/abs/0706.3740}.

\bibitem[Lee08]{Lee07b}
\bysame, \emph{Banach spaces with polynomial numerical index 1}, Bull. Lond.
  Math. Soc. \textbf{40} (2008), 193--198.

\bibitem[Lim78]{Lim78} 
{\AA}svald Lima, \emph{Intersection properties of balls in spaces of compact
              operators}, Ann. Inst. Fourier (Grenoble), \textbf{28} (1978), 35 -- 65.

\bibitem[Lin63]{Lin63}
Joram Lindenstrauss, \emph{On operators which attain their norm}, Israel J.
  Math. \textbf{1} (1963), 139--148.

\bibitem[Pal04]{Pal04}
{\'A}ngel~Rodr{\'{\i}}guez Palacios, \emph{Numerical ranges of uniformly
  continuous functions on the unit sphere of a {B}anach space}, J. Math. Anal.
  Appl. \textbf{297} (2004), no.~2, 472--476, Special issue dedicated to John
  Horv\'ath.

\bibitem[Rei91]{Rei91}
Shlomo Reisner, \emph{Certain {B}anach spaces associated with graphs and
  {CL}-spaces with {$1$}-unconditional bases}, J. London Math. Soc. (2)
  \textbf{43} (1991), no.~1, 137--148.

\end{thebibliography}

{\noindent Department of Mathematics\\
POSTECH\\
Republic of Korea}

\end{document}